\documentclass[11pt]{amsart}

\usepackage{amsmath,amsthm,amsfonts,amssymb,mathrsfs}
\usepackage[margin=1in]{geometry}
\usepackage[bookmarks]{hyperref}
\usepackage{verbatim}
\usepackage{todonotes}

\usepackage{fancyvrb }
\usepackage{color}

\newenvironment{enumalph}{\begin{enumerate}  }{\end{enumerate}}

\newtheorem{theorem}{Theorem}[subsection]

\newtheorem{proposition}[theorem]{Proposition}

\newtheorem{conjecture}[theorem]{Conjecture}

\newtheorem{corollary}[theorem]{Corollary}

\newtheorem{lemma}[theorem]{Lemma}
\input xy
\xyoption{all}

\theoremstyle{definition}

\newtheorem{definition}[theorem]{Definition}

\newtheorem{example}[theorem]{Example}

\newtheorem{remark}[theorem]{Remark}

\DeclareMathOperator{\ch}{ch}

\DeclareMathOperator{\codim}{codim}

\DeclareMathOperator{\opH}{H}

\DeclareMathOperator{\ind}{ind}
\DeclareMathOperator{\Lie}{Lie}

\DeclareMathOperator{\Tr}{Tr}

\newcommand{\calO}{\mathcal{O}}

\newcommand{\A}{\mathbb{A}}

\newcommand{\calF}{\mathcal{F}}

\newcommand{\spec}{\mbox{{\rm{Spec}}\;}}
\newcommand{\red}{\rm{red}}
\newcommand{\reg}{\rm{reg}}
\newcommand{\subreg}{\rm{subreg}}
\newcommand{\sub}{\rm{sub}}

\newcommand{\calL}{\mathcal{L}}

\newcommand{\frakb}{\mathfrak{b}}
\newcommand{\g}{\mathfrak{g}}

\newcommand{\fraku}{\mathfrak{u}}

\newcommand{\fraksl}{\mathfrak{sl}}

\newcommand{\N}{\mathcal{N}}

\newcommand{\Fp}{\mathbb{F}_p}

\begin{document}

\title[Commuting varieties of $r$-tuples over Lie algebras]{Commuting varieties of $r$-tuples over \\ Lie algebras}

\author{Nham V. Ngo}
\address{Department of Mathematics, Statistics, and Computer Science\\ University of Wisconsin-Stout \\ Menomonie\\ WI~54751, USA}

\address{{\bf Current address:} Department of Mathematics and Statistics \\ Lancaster University \\ Lancaster \\ LA1 4YW, UK}
\email{n.ngo@lancaster.ac.uk}


\maketitle

\begin{abstract}
Let $G$ be a simple algebraic group defined over an algebraically closed field $k$ of characteristic $p$ and let $\g$ be the Lie algebra of $G$. It is well known that for $p$ large enough the spectrum of the cohomology ring for the $r$-th Frobenius kernel of $G$ is homeomorphic to the commuting variety of $r$-tuples of elements in the nilpotent cone of $\g$ [Suslin-Friedlander-Bendel,  J. Amer. Math. Soc, \textbf{10} (1997), 693--728]. In this paper, we study both geometric and algebraic properties including irreducibility, singularity, normality and Cohen-Macaulayness of the commuting varieties $C_r(\mathfrak{gl}_2), C_r(\fraksl_2)$ and $C_r(\N)$ where $\N$ is the nilpotent cone of $\fraksl_2$. Our calculations lead us to state a conjecture on Cohen-Macaulayness for commuting varieties of $r$-tuples. Furthermore, in the case when $\g=\fraksl_2$, we obtain interesting results about commuting varieties when adding more restrictions into each tuple. In the case of $\fraksl_3$, we are able to verify the aforementioned properties for $C_r(\fraku)$. Finally, applying our calculations on the commuting variety $C_r(\overline{\calO_{\sub}})$ where $\overline{\calO_{\sub}}$ is the closure of the subregular orbit in $\fraksl_3$, we prove that the nilpotent commuting variety $C_r(\N)$ has singularities of codimension $\ge 2$.     
\end{abstract}

\section{Introduction}\label{introduction}
\subsection{} Let $k$ be an algebraically closed field of characteristic $p$ (possibly $p=0$). For a Lie algebra $\g$ over $k$ and a closed subvariety $V$ of $\g$, the commuting variety of $r$-tuples over $V$ is defined as the collection of all $r$-tuples of pairwise commuting elements in $V$. In particular, for each positive integer $r$, we define  
\[ C_r(V)=\{(v_1,\ldots,v_r)\in V^r~|~[v_i,v_j]=0~\text{for all}~1\le i\le j\le r\}.\] 
When $r=2$, we call $C_r(V)$ an ordinary commuting variety (cf. \cite{Vas:1994}). Let $\N$ be the nilpotent cone of $\g$. We then call $C_r(\N)$ the nilpotent commuting variety of $r$-tuples. For simplicity, from now on we will call $C_r(V)$ a commuting variety over $V$ whenever $r>2$ in order to distinguish it from ordinary commuting varieties.

The variety $C_2(\g)$ over a reductive Lie algebra $\g$ was first proved to be irreducible in characteristic $0$ by Richardson in 1979 \cite{R:1979}. In positive characteristics, Levy showed that $C_2(\g)$ is irreducible under certain mild assumptions on $G$ \cite{Le:2002}. In 2003, Premet completely determined irreducible components of the nilpotent commuting variety $C_2(\N)$ for an arbitrary reductive Lie algebra in characteristic $0$ or $p$, a good prime for the root system of $G$. In particular, he showed that the nilpotent commuting variety $C_2(\N)$ is equal to the union of irreducible components $\mathcal{C}(e_i)=\overline{G\cdot(e_i,\mbox{Lie}(Z_{[G,G]}(e_i)))}$, where the $e_i$'s are representatives for distinguished nilpotent orbits in $\g$ (which contain elements whose centralizers are in the nilpotent cone).

Up to now, the properties of being Cohen-Macaulay and normal for ordinary commuting varieties are in general not verified \cite{K:2005}. There is a long-standing conjecture stating that the commuting variety $C_2(\g)$ is always normal (see \cite{Po:2008} and \cite{Pr:2003}). Artin and Hochster claimed that $C_2(\mathfrak{gl}_n)$ is a Cohen-Macaulay integral domain (cf. \cite{K:2005},\cite{MS:2011}). This is verified up to $n=4$ by the computer program Macaulay \cite{H:2006}. There is not much hope for verifying Cohen-Macaulayness of nilpotent commuting varieties since their defining ideals are not radical, thus creating great difficulties for computer calculations. As pointed out by Premet, all components of $C_2(\N)$ share the origin $0$, so if it is reducible then it can never be normal. 

Not much is known about commuting varieties. In case $V=\mathfrak{gl}_n$, commuting varieties were studied by Gerstenhaber, Guralnick-Sethuraman, and Kirillov-Neretin. Gerstenhaber proved that $C_r(\mathfrak{gl}_n)$ is reducible for all $n\ge 4$ and $r\ge 5$ \cite{G:1961}. In 1987, Kirillov and Neretin lowered the bound on $r$ by simply showing that $C_4(\mathfrak{gl}_4)$ is reducible. Moreover, they proved that when $n=2$ or $3$ the commuting variety $C_r(\mathfrak{gl}_n)$ is irreducible for all $r\ge 1$ \cite{KN:1987}. In 2000, Guralnick and Sethuraman studied the case when $r=3$ and concluded that $C_3(\mathfrak{gl}_n)$ can be either irreducible (for $n\le 10$) or reducible (for $n\ge 30$) \cite{GS:2000}, \cite{S:2012}. In general, the study of $C_r(V)$ for arbitrary $r$ remains mysterious even in simple cases.

\subsection{Main results}\label{main results} The results in the present paper were motivated by investigating the cohomology for Frobenius kernels of algebraic groups. To be more precise, let $G$ be an algebraic group defined over $k$, and let $G_r$ be the $r$-th Frobenius kernel of $G$. Then there is a homeomorphism between the maximal ideal spectrum of the cohomology ring for $G_r$ and the nilpotent commuting variety over the Lie algebra Lie$(G)$ whenever the characteristic $p$ is large enough \cite{BFS1:1997, BFS2:1997}. In this paper, we tackle the irreducibility, normality and Cohen-Macaulayness of this variety for simple cases. The main results are summarized in the following.
\begin{theorem}
Suppose $p>2$. Then for each $r\ge 1$, the commuting varieties $C_r(\mathfrak{gl}_2), C_r(\fraksl_2)$ and $C_r(\N)$ are irreducible, normal and Cohen-Macaulay.
\end{theorem}
For readers' convenience, we sketch the structure of the paper as follows. We first introduce notation, terminology and conventions in Sections \ref{notation} and \ref{background}. We also prove several properties related to the map $m:G\times^BC_r(\fraku)\to C_r(\N)$ where $\fraku$ is the Lie algebra of the unipotent radical of a fixed Borel subgroup $B$ of $G$, and $\N$ is the nilpotent cone of $\g$. For instance, we show that the map is surjective and satisfies the hypotheses of Zariski's Main Theorem. These results are analogous to those for the moment map from $G\times^B\fraku$ to $\N$. In the next section, we first show a connection between $C_r(\mathfrak{gl}_n)$ and $C_r(\fraksl_n)$ for arbitrary $n, r\ge 1$. This link reduces all of the works for $C_r(\mathfrak{gl}_n)$ to that for $C_r(\fraksl_n)$. Then we consider the case $n=2$ and prove the properties of being irreducible, normal, and Cohen-Macaulay for both varieties by exploiting a fact of determinantal rings. In addition, the analogs are shown for the nilpotent commuting variety over rank 2 matrices in Section \ref{nilpotent commuting variety}. It should be noticed that the nilpotency condition makes the defining ideal of this variety non-radical, thus creating more difficulties in our task. In the case when $\g=\fraksl_2$, to obtain the Cohen-Macaulayness of $C_r(\N)$, we first analyze the geometry by intersecting $C_r(\N)$ with a hypersurface and reduce the problem to showing that a certain class of ideals are radical. Then we prove that this family of ideals belongs to a principal radical system, a deep concept in commutative algebra introduced by Hochster \cite{BV:1988}. As a consequence, we show that the moment map $m:G\times^B\fraku^r\rightarrow C_r(\N)$ admits rational singularities (cf. Proposition \ref{rational resolution}). 

As an application we can compute the characters of the coordinate algebra for this variety (cf. Theorem \ref{characters of coordinate alg of comm var for sl_2}). Combining this with an explicit calculation for the reduced $G_r$-cohomology ring in \cite{N:2012}, we obtain an alternative proof for the main result in \cite{BFS1:1997} for the case when $G=SL_2$.

In Section \ref{mixed cases}, we study commuting varieties with additional restrictions on each tuple. In particular, let $V_1,\ldots, V_r$ be closed subvarieties of a Lie algebra $\g$. Define
\[ C(V_1,\ldots, V_r)=\{(v_1,\ldots, v_r)\in V_1\times\cdots\times V_r~|~[v_i,v_j]=0~,~1\le i\le j\le r \}, \]   
a mixed commuting variety over $V_1,\ldots,V_r$. In the case where $\g=\fraksl_2$, we can explicitly describe the irreducible decomposition for any mixed commuting variety over $\g$ and $\N$. This shows that such varieties are mostly not Cohen-Macaulay or normal. 
  
Section \ref{comvar of 3 by 3} involves results about the geometric structure of nilpotent commuting varieties over various sets of 3 by 3 matrices. In particular, let $\N$ be the nilpotent cone of $\fraksl_3$. We first study the variety $C_r(\fraku)$ and then obtain the irreducibility and dimension of $C_r(\N)$. Next we apply our calculations on $C_r(\overline{\calO_{\sub}})$ to classify singularities of $C_r(\N)$ and show that they are in codimension greater than or equal to 2 (cf. Theorem \ref{codim of sing}), here $\overline{\calO_{\sub}}$ is the closure of the subregular orbit in $\N$. This result indicates that the variety $C_r(\N)$ satisfies the necessary condition $(R1)$ of Serre's criterion for normality. 

\section{Notation}\label{notation}

\subsection{Root systems and combinatorics}\label{combinatorics} Let $k$ be an algebraically closed field of characteristic $p$. Let $G$ be a simple, simply-connected algebraic group over $k$, defined and split over the prime field $\Fp$. Fix a maximal torus $T \subset G$, also split over $\Fp$, and let $\Phi$ be the root system of $T$ in $G$. Fix a set $\Pi = \{ \alpha_1,\ldots,\alpha_n \}$ of simple roots in $\Phi$, and let $\Phi^+$ be the corresponding set of positive roots. Let $B \subseteq G$ be the Borel subgroup of $G$ containing $T$ and corresponding to the set of negative roots $\Phi^-$, and let $U \subseteq B$ be the unipotent radical of $B$. Write $U^+ \subseteq B^+$ for the opposite subgroups. Set $\g = \Lie(G)$, the Lie algebra of $G$, $\frakb = \Lie(B)$, $\fraku = \Lie(U)$.

\subsection{Nilpotent orbits} We will follow the same conventions as in \cite{Hum:1995} and \cite{Jan:2004}. Given a $G$-variety $V$ and a point $v$ of $V$, we denote by $\calO_v$ the $G$-orbit of $v$ (i.e., $\calO_v=G\cdot v$). For example, consider the nilpotent cone $\N$ of $\g$ as a $G$-variety with the adjoint action. There are well-known orbits: $\calO_{\reg}=G\cdot v_{\reg}, \calO_{\subreg}=G\cdot v_{\subreg}$, (we abbreviate it by $\calO_{\sub}$,) and $\calO_{\rm{min}}=G\cdot v_{\rm{min}}$ where $v_{\reg}, v_{\subreg},$ and $v_{\rm{min}}$ are representatives for the regular, subregular, and minimal orbits. Denote by $z(v)$ and $Z(v)$ respectively the centralizers of $v$ in $\g$ and $G$. It is well-known that $\dim z(v)=\dim Z(v)$ and $\dim\calO_v=\dim G-\dim z(v)$. For convenience, we write $z_{\reg}$ ($z_{\sub}$ and $z_{\min}$) for the centralizers of $v_{\reg}$ ($v_{\sub}$ or $v_{\min}$). It is also useful to keep in mind that every orbit is a smooth variety. Sometimes, we use $\calO_V$ for the structure sheaf of a variety $V$ (see \cite{CM:1993}, \cite{Hum:1995} for more details).

\subsection{Basic algebraic geometry conventions}\label{algebraic geometry conventions} Let $R$ be a commutative Noetherian ring with identity. We use $R_{\red}$ to denote the reduced ring $R/\sqrt{0}$ where $\sqrt{0}$ is the radical ideal of the trivial ideal $0$, which consists of all nilpotent elements of $R$. Let $\spec R$ be the spectrum of all prime ideals of $R$. If $V$ is a closed subvariety of an affine space $\A^n$, we denote by $I(V)$ the radical ideal of $k[\A^n]=k[x_1,\ldots,x_n]$ associated to this variety.

Given a $G$-variety $V$, $B$ act freely on $G\times V$ by setting $b\cdot(g,v)=(gb^{-1},bv)$ for all $b\in B,g\in G$ and $v\in V$. The notation $G\times^BV$ stands for the fiber bundle associated with the projection $\pi:G\times^BV\rightarrow G/B$ with fiber $V$. Topologically, $G\times^BV$ is the quotient space of $G\times V$ in which the equivalence relation is given as 
\[ (g,v)\sim (g',v')\Leftrightarrow (g',v')=b\cdot(g,v)~~~\text{for some}~~b\in B.\]
In other words, each equivalence class of $G\times^BV$ represents a $B$-orbit in $G\times V$. The map $m:G\times^BV\rightarrow G\cdot V$ defined by mapping each equivalence class $[g,v]$ to the element $g\cdot v$ for all $g\in G,v\in V$ is called the moment morphism. It is obviously surjective. Let $X$ be an affine variety. Then we always write $k[X]$ for the coordinate ring of $X$ which is the same as the ring of global sections $\calO_X(X)$. Although $G\times^BV$ is not affine, we still denote $k[G\times^BV]$ for its ring of global sections. It is sometimes useful to make the identification $k[G\times^BV]\cong k[G\times V]^B$.

Let $f:X\rightarrow Y$ be a morphism of varieties. Denote by $f_*$ the direct image functor from the category of sheaves over $X$ to the category of sheaves over $Y$. One can see that this is a left exact functor. Hence, we have the right derived functors of this direct image. We call these functors higher direct images and denote them by $R^if_*$ with $i> 0$. In particular, if $Y=\spec A$ is affine and $\calF$ is a quasi-coherent sheaf on $X$, then we have $R^if_*(\calF)\cong\calL(\opH^i(X,\calF))$ where $\calL$ is the exact functor mapping an $A$-module $M$ to its associated sheaf $\calL(M)$. Here we follow conventions in \cite{H:1977}.    

\section{Commutative algebra and Geometry}\label{background}

In this section we introduce concepts in commutative algebra and geometry that play important roles in the later sections. In particular, we recall the definition of Cohen-Macaulay varieties and their properties. We also show that the moment map $G\times^BC_r(\fraku)\to C_r(\N)$ is always surjective for arbitrary type of $G$, and that it is a proper birational morphism. We then review a number of well-known results in algebraic geometry. 

\subsection{Cohen-Macaulay Rings} We first define regular sequences, which are the key ingredient in the definition of Cohen-Macaulay rings. Readers can refer to \cite{E:1995} and \cite{H:1977} for more details.

\begin{definition}
Let $R$ be a commutative ring and let $M$ be an $R$-module. A sequence $x_1,\ldots,x_n\in R$ is called a regular sequence on $M$ (or an $M$-sequence) if it satisfies\\
(1) $(x_1,\ldots,x_n)M\ne M$, and \\
(2) for each $1\le i\le n$, $x_i$ is not a zero-divisor of $M/(x_1,\ldots,x_{i-1})M$. 

\end{definition}

Consider $R$ as a left $R$-module. For a given ideal $I$ of $R$, it is well-known that the length of any maximal regular sequence in $I$ is unique. It is called the {\it depth} of $I$ and denoted by $\text{depth}(I)$. The {\it height} or {\it codimension} of a prime ideal $J$ of $R$ is the supremum of the lengths of chains of prime ideals descending from $J$. Equivalently, it is defined as the Krull dimension of $R/J$. In particular, if $R$ is an integral domain that is finitely generated over a field, then $\codim(J)=\dim R-\dim(R/J)$. We are now ready to define a Cohen-Macaulay ring. 

\begin{definition}
A ring $R$ is called Cohen-Macaulay if $\text{depth}(I)=\codim(I)$ for each maximal ideal $I$ of $R$. A variety $V$ is called Cohen-Macaulay if its coordinate ring $k[V]$ is a Cohen-Macaulay ring.
\end{definition} 

\begin{example} Smooth varieties are Cohen-Macaulay. The nilpotent cone $\N$ of a simple Lie algebra $\g$ over an algebraically closed field of good characteristic is Cohen-Macaulay \cite[8.5]{Jan:2004}.
\end{example}

\subsection{Minors and Determinantal rings}
Let $U = (u_{ij})$ be an $m\times n$ matrix over a ring $R$. For indices $a_1, \ldots , a_t, b_1, \ldots , b_t$ such that $1\le a_i\le m, 1\le b_i\le n, i = 1, \ldots , t$, we put
\[ [a_1, \ldots , a_t ~|~ b_1, \ldots , b_t] = \det
\left( \begin{array}{ccc}
u_{a_1b_1} & \cdots & u_{a_1b_t} \\
\vdots & \ddots & \vdots \\
u_{a_tb_1} & \cdots & u_{a_tb_t} \end{array} \right). \]
We do not require that $a_1, \ldots , a_t$ and $b_1, \ldots , b_t$ are given in ascending order. Note that
\[ [a_1,\ldots, a_t~|~b_1,\ldots, b_t] = 0 \]
if $t > \min(m, n)$. To be convenient, we let $[\emptyset~|~\emptyset] = 1$. If $a_1\le\cdots\le a_t$ and $b_1\le\cdots\le b_t$ we call $[a_1, \ldots , a_t~|~b_1, \ldots , b_t]$ a $t$-minor of $U$.

\begin{definition}
Let $B$ be a commutative ring, and consider an $m\times n$ matrix
\[ X=\left( \begin{array}{ccc}
x_{11} & \cdots & x_{1n} \\
\vdots & \ddots & \vdots \\
x_{m1} & \cdots & x_{mn} \end{array} \right) \]
whose entries are independent indeterminates over $R$. Let $R(X)$ be the polynomial ring over all the indeterminates of $X$, and let $I_t(X)$ be the ideal in $R(X)$ generated by all $t$-minors of $X$. For each $t\ge 1$, the ring 
\[ R_t(X)=\frac{R(X)}{I_t(X)} \]
is called a {\it determinantal ring}.
\end{definition}

For readers' convenience, we recall nice properties of determinantal rings as follows.

\begin{proposition}\cite[1.11, 2.10, 2.11, 2.12]{BV:1988}\label{determinantal rings}
If $R$ is a reduced ring, then for every $1\le t\le\min(m,n)$, $R_t(X)$ is a reduced, Cohen-Macaulay, normal domain of dimension $(t-1)(m+n-t+1)$.
\end{proposition}

\subsection{The moment morphism}\label{moment map}
Suppose $G$ is a simple algebraic group. It is well-known that $\N=G\cdot\fraku$ where $\fraku$ is the Lie algebra of the unipotent radical subgroup $U$ of the Borel subgroup $B$ of $G$, and the dot is the adjoint action of $G$ on the Lie algebra $\g$. Note that if $u_1,u_2$ are commuting in $\fraku$, then so are $g\cdot u_1,g\cdot u_2$ in $\N$ for each $g\in G$. This observation can be generalized to give the following moment map
\[ m:G\times^BC_r(\fraku)\rightarrow C_r(\N)\]
by setting $m[g,(u_1,\ldots,u_r)]=(g\cdot u_1,\ldots,g\cdot u_r)$ for all $g\in G,$ and $(u_1,\ldots,u_r)\in C_r(\fraku)$. In the case when $r=1$, this is the moment map in the {\it Springer resolution}. Therefore, we also call it {\it the moment morphism} for each $r\ge 1$. The following proposition shows surjectivity of this morphism.

\begin{theorem}\label{surjectivity of moment}\footnote{The author would like to thank Christopher M. Drupieski for the main idea in this proof}
The moment morphism $m:G\times^BC_r(\fraku)\rightarrow C_r(\N)$ is always surjective.
\end{theorem}

\begin{proof}
Suppose $(v_1,\ldots,v_r)\in C_r(\N)$. Let $\mathfrak{b}'$ be the vector subspace of $\g$ spanned by the $v_i$. As $[v_i,v_j]=0$ for all $1\le i,j\le r$, $\mathfrak{b}'$ is an abelian, hence solvable, Lie subalgebra of $\g$. Thus, there exists a maximal solvable subalgebra $\mathfrak{b}''$ of $\g$ containing $\mathfrak{b}'$. By \cite[Theorem 16.4]{Hum:1978}, $\mathfrak{b}''$ and our Borel subalgebra $\mathfrak{b}$ are conjugate under some inner automorphism $Ad(g)$ with $g\in G$. So there exist $u_1,\ldots,u_r\in\mathfrak{b}$ such that 
\[(v_1,\ldots,v_r)=Ad(g^{-1})(u_1,\ldots,u_r)=g^{-1}\cdot(u_1,\ldots,u_r)=m[g^{-1},(u_1,\ldots,u_r)]. \] 
As all the $v_i$ are nilpotent and commuting, so are the $u_i$. This shows that $m$ is surjective.
\end{proof}
As a corollary, we establish the connection between irreducibility of $C_r(\fraku)$ and $C_r(\N)$.

\begin{theorem}\label{C(u)-and-C(N)}
For each $r\ge 1$, if $C_r(\fraku)$ is irreducible then so is $C_r(\N)$.
\end{theorem}

\begin{proof}
As the moment morphism $G\times C_r(\fraku)\to C_r(\N)$ is surjective and $G$ is irreducible, the irreducibility of $C_r(\N)$ follows from that of $C_r(\fraku)$.
\end{proof}

\subsection{Zariski's Main Theorem}
Zariski's Main Theorem is one of the powerful tools to study structure sheaves of two schemes. In this subsection, we state the version for varieties and show that the moment map in the preceding subsection satisfies the hypotheses of Zariski's Main Theorem. We first look at proper morphisms. As defining properness requires terminology from algebraic geometry, we refer readers to \cite[Section II.4]{H:1977} for the details. We only introduce some important characterizations of proper morphisms which will be useful later.

\begin{proposition}\label{proper maps} In the following properties, all the morphisms are taken over Noetherian schemes. 
\begin{enumalph}
\item A closed immersion is proper.
\item The composition of two proper morphims is proper.
\item A projection $X \times Y \rightarrow X$ is proper if and only if $Y$ is projective.
\end{enumalph}
\end{proposition}

We also recall that a rational map $\varphi:X\to Y$ (which is a morphism only defined on some open subset)  is called birational if it has an inverse rational map.

There are many versions of Zariski's Main Theorem. Here we state a ``pre-version" of the theorem since the Main Theorem immediately follows from this result (cf. \cite[Corollary III.11.3 and III.11.4]{H:1977}).

\begin{theorem}\label{Zariski}
Let $f:X\rightarrow Y$ be a birational proper morphism of varieties and suppose $Y$ is normal. Then $f_*\calO_X=\calO_Y$.
\end{theorem}

We now verify that the morphism $m:G\times^BC_r(\fraku)\to C_r(\N)$ satisfies the hypotheses of Zariski's Main Theorem. In other words, we have
\begin{proposition}\label{birational moment}
For each $r\ge 1$, the moment morphism $m:G\times^BC_r(\fraku)\to C_r(\N)$ is birational proper.
\end{proposition}
\begin{proof}
We generalize the proofs of Lemmas 1 and 2 in \cite[6.10]{Jan:2004}. For the properness, we consider the map 
\[ \epsilon:G\times^BC_r(\fraku)\hookrightarrow G/B\times C_r(\N)\]
with $\epsilon[g,(u_1,\ldots,u_r)]=(gB,g\cdot(u_1,\ldots,u_r))$ for all $g\in G$ and $(u_1,\ldots,u_r)\in C_r(\fraku)$. By the same argument as in \cite[6.4]{Jan:2004}, we can show that this map is a closed embedding; hence a proper morphism by Proposition \ref{proper maps}(a). Next, as $G/B$ is projective, the projection map $p:G/B\times C_r(\N)\to C_r(\N)$ is proper by Proposition \ref{proper maps}(c). Therefore, part (b) of Proposition \ref{proper maps} implies that $m=p\circ\epsilon$ is also proper.

Consider the projection of $C_r(\N)$ onto the first factor $p_1:C_r(\N)\to\N$. Recall that $z_{\reg}$ is the centralizer of a fixed regular element $v_{\reg}$ in $\N$. From Lemma 35.6.7 in \cite{TY:2005}, $z_{\reg}$ is a commutative Lie algebra. Then we have 
\[p_1^{-1}(\calO_{\reg})=C(\calO_{\reg},\N,\ldots,\N)=G\cdot(v_{\reg},z_{\reg},\ldots,z_{\reg}). \]
As $\calO_{\reg}$ is an open subset in $\N$, the preimage $p_1^{-1}(\calO_{\reg})$ is open in $C_r(\N)$. Let $V=p_1^{-1}(\calO_{\reg})$. Since $Z_G(v_{\reg})\subseteq B$, we have $m$ induces an isomorphism from $m^{-1}(V)$ onto $V$. It follows that $m$ is a birational morphism. 
\end{proof}

\begin{remark}
We have not shown that $m$ satisfies all the hypotheses of Zariski's Theorem since the normality for $C_r(\N)$ is still unkown. As mentioned in Section \ref{introduction}, the variety $C_r(\N)$ is normal only when it is irreducible. In particular, Premet already proved that $C_2(\N)$ is irreducible if and only if $G$ is of type $A$ \cite{Pr:2003}. By considering the natural projection map $C_r(\N)\to C_2(\N)$, it follows that if $G$ is not of type $A$ then $C_r(\N)$ is reducible for each $r\ge 2$. In this paper we prove for arbitrary $r\ge 1$ that the variety $C_r(\N)$ is irreducible for types $A_1$ and $A_2$. The result for type $A_n$ with arbitrary $n, r>2$ remains an open problem.
\begin{conjecture}
If $G$ is of type $A$, then $C_r(\N)$ is irreducible. Moreover, it is normal. In other words, the morphism
\[ m:G\times^BC_r(\fraku)\to C_r(\N) \]
satifies all the hypotheses of Zariski's Theorem.
\end{conjecture} 
\end{remark}

\subsection{Singularities and Resolutions}
Here we state an observation on determinning the singular points of an affine variety defined by homogeneous polynomials, and then define a resolution of singularities.  

\begin{proposition}\label{trivial singular}
Let $V$ be an affine variety whose defining radical ideal is generated by a non-empty set of homogeneous polynomials of degree at least 2. Then $0$ is always a singular point of $V$.
\end{proposition}

\begin{proof}
Suppose $V$ is an affine subvariety of the affine space $\A^m$ associated with the coordinate ring $k[x_1,\ldots,x_m]$. Let $f_1,\ldots,f_n$ be the set of polynomials defining $V$. Note that $\dim V< m$ since $n\ge 1$. Consider the Jacobian matrix 
\[ \left( \begin{array}{ccc} \frac{df_1}{x_1} & \cdots & \frac{df_1}{x_m} \\
\vdots & \ddots & \vdots \\
\frac{df_n}{x_1} & \cdots & \frac{df_n}{x_m} \end{array} \right) \]
As all the $f_i$ are homogeneous of degree $\ge 2$, we have $\frac{df_i}{x_j}(0,\ldots,0)=0$ for all $1\le i\le n, 1\le j\le m$. It follows that the tangent space at 0 has dimension $m$ which is greater than $\dim V$. Thus $0$ is a singular point of $V$.
\end{proof}

\begin{definition}
A variety $X$ has a resolution of singularities if there exists a non-singular variety $Y$ such that there is a proper birational morphism from $Y$ to $X$.
\end{definition}

\begin{definition}
A variety $X$ has rational singularities if it is normal and has a resolution of singularities
\[f:Y\to X\]
such that the higher direct image $R^if_*\calO_Y$ vanishes for $i\ge 1$. (Sometimes one calls $f$ a rational resolution.)
\end{definition}

In Lie theory, the nullcone $\N$ admits a resolution of singularites, the Springer resolution, and also has rational singularities. One of our goals in the present paper is to generalize this resolution for the nilpotent commuting variety over rank two Lie algebra.

\section{Commuting Varieties over 2 by 2 matrices}\label{Comvar of 2 by 2}

\subsection{} Recall that the problem of showing that the commuting variety over $n\times n$ matrices is Cohen-Macaulay and normal is very difficult to verify. With ordinary commuting varieties, computer verification works up to $n=4$ \cite{H:2006}. There are also some studies on the Cohen-Macaulayness of other structures closely similar to ordinary commuting varieties by Knutson, Mueller, Zolbanin-Snapp and Zoque (cf. \cite{K:2005},\cite{Mu:2007}, \cite{MS:2011}, \cite{Zoque:2010}). Very little appears to be known for commuting varieties in general. In this section, we confirm the properties of being Cohen-Macaulay and normal for $C_r(\mathfrak{gl}_2)$ and $C_r(\fraksl_2)$ with arbitrary $r\ge 1$.

\subsection{Nice properties of $C_r(\mathfrak{gl}_2)$ and $C_r(\fraksl_2)$}
We first show a general result connecting the commuting varieties over $\mathfrak{gl}_n$ and $\fraksl_n$.

\begin{theorem}\label{isomorphism of comm vars}\footnote{The author would like to thank William Graham for his assistance in generalizing the result to $\mathfrak{gl}_n$.}
For each $n$ and $r\ge 1$, if $p$ does not divide $n$, then there is an isomorphism of varieties from $C_r(\mathfrak{gl}_n)$ to $C_r(\fraksl_n)\times\mathbb{A}^r$ defined by setting 
\begin{align}\label{C_r(M_2)}
\varphi: (v_1,\ldots,v_r) &\mapsto \left(v_1-\frac{\Tr(v_1)}{n}I_n,\ldots,v_r-\frac{\Tr(v_r)}{n}I_n\right)\times\left(\Tr(v_1),\ldots,\Tr(v_r)\right)
\end{align}
for $v_i\in\mathfrak{gl}_n$.
\end{theorem}

\begin{proof}
It is easy to see that adding or subtracting $cI_n$ from the $v_i$ does not change the commuting conditions on the $v_i$. So the morphism $\varphi$ in \eqref{C_r(M_2)} is well-defined and its inverse is
\[ \varphi^{-1}:(u_1,\ldots,u_r)\times (a_1,\ldots,a_r)\mapsto \left(u_1+\frac{a_1}{n}I_n,\ldots,u_r+\frac{a_r}{n}I_n\right). \]
This completes the proof.
\end{proof}

This result implies that our work for $C_r(\mathfrak{gl}_2)$ will be done if we can prove that $C_r(\fraksl_2)$ is Cohen-Macaulay. Notice that Popov proved the normality of this variety in the case $r=2$ \cite[1.10]{Po:2008}. However, his proof depends on computer calculations to verify that the defining ideal is radical. Here we propose another approach that completely solves the problem for arbitrary $r$. Let $\fraksl_2^r$ be the affine space defined as
\[ \left\{ \left( \begin{array}{cc} x_1 &  y_1 \\ z_1 & -x_1\end{array} \right),\ldots,\left( \begin{array}{cc} x_r &  y_r \\ z_r & -x_r\end{array} \right)~\mid~x_i,\ y_i,\ z_i\in k, 1\le i\le r \right\}. \]
Then the variety $C_r(\fraksl_2)$ can be defined as the zero locus of the following ideal
\begin{align}\label{commutator of C_r(sl2)}
 J=\langle x_iy_j-x_jy_i \ , \ y_iz_j-y_jz_i \ , \ x_iz_j-x_jz_i~ \mid ~1\le i\le j\le r \rangle.
\end{align}

\begin{proposition}\label{normal-of-C_r(sl)}
For each $r\ge 1$, the variety $C_r(\fraksl_2)$ is:
\begin{enumalph}
\item irreducible of dimension $r+2$,
\item Cohen-Macaulay and normal.
\end{enumalph}
\end{proposition}

\begin{proof}
It was shown by Kirillov and Neretin that $C_r(\mathfrak{gl}_2)$ and $C_r(\mathfrak{gl}_3)$ are irreducible for all $r$ in characteristic 0, \cite[Theorem 4]{KN:1987}. The irreducibility of $C_r(\fraksl_2)$ in characteristic $0$ then follows by Theorem \ref{isomorphism of comm vars}. Here we provide a general proof for any characteristic $p\ne 2$.

Consider the ideal $J$ above as the ideal $I_2(\mathcal{X})$ generated by all 2-minors of the following matrix
\[\mathcal{X}=\left( \begin{array}{cccc} x_1 &  x_2 & \cdots & x_r \\y_1 &  y_2 & \ddots & y_r\\z_1 &  z_2 & \cdots & z_r\end{array} \right). \]
Then we can identify $k[C_r(\fraksl_2)]$ with the ring $R_2(\mathcal{X})=\frac{k(\mathcal{X})}{I_2(\mathcal{X})}$. It follows immediately from Proposition \ref{determinantal rings} that $R_2(\mathcal{X})$ is a Cohen-Macaulay and normal domain, hence completing the proof.
\end{proof}

\begin{corollary}\label{normality of C_r(M_2)}
Suppose $p\ne 2$. Then for each $r\ge 1$, the variety $C_r(\mathfrak{gl}_2)$ is
\begin{enumalph}
\item irreducible of dimension $2r+2$,
\item Cohen-Macaulay and normal. 
\end{enumalph}
\end{corollary}

\begin{proof}
Follows immediately from Theorem \ref{isomorphism of comm vars} and Proposition \ref{normal-of-C_r(sl)}.
\end{proof}

This computation allows us to state a conjecture about Cohen-Macaulayness for commuting varieties which is a generalization of that for ordinary commuting varieties.

\begin{conjecture}
Suppose $p\nmid n$. Then both commuting varieties $C_r(\mathfrak{gl}_n)$ and $C_r(\fraksl_n)$ are Cohen-Macaulay.
\end{conjecture}

\section{Nilpotent Commuting Varieties over $\fraksl_2$}\label{nilpotent commuting variety}

With the nilpotency condition, problems involving commuting varieties turn out to be more difficult. The irreducibility of ordinary nilpotent commuting varieties was studied by Baranovsky, Premet, Basili and Iarrobino (cf. \cite{Ba:2001}, \cite{Pr:2003}, \cite{Ba:2007}, \cite{BI:2008}). However, there has not been any successful work on normality and Cohen-Macaulayness even in simple cases. In this section, we completely prove this conjecture for the nilpotent commuting variety over $\fraksl_2$. 

\subsection{Irreducibility} Let $G=SL_2, \g=\fraksl_2$, and $k$ be an algebraically closed field of characteristic $p\ne 2$. The nilpotent cone of $\g$ then can be written as follows.
\begin{align*}
\N =\left\{ \left( \begin{array}{cc} x & y \\
                             z & -x \end{array} \right) |~ x^2+yz=0~ \mbox{with}~x,y,z\in k\right\}
\end{align*}
Note that for each $r\ge 1$, $C_r(\fraku)=\fraku^r$, so that the moment map in Section \ref{moment map} can be rewritten as
\begin{equation}\label{moment-sl2}
 m:G\times^B\fraku^r\rightarrow C_r(\N).
\end{equation}
As $G/B$ and $\fraku^r$ are smooth varieties, so is the vector bundle $G\times^B\fraku^r$. In addition, the moment map $m$ is known to be proper birational from Proposition \ref{birational moment}. It follows that $G\times^B\fraku^r$ is a resolution of singularities for $C_r(\N)$ through the morphism $m$. We will see later in this section that $m$ is in fact a rational resolution. First we study geometric properties of $C_r(\N)$.  

\begin{proposition}\label{properties-of-C_r}
For every $r\ge 1$, we have
\begin{enumalph}
\item $C_r(\N)$ is an irreducible variety of dimension $r+1$.
\item The only singular point in $C_r(\N)$ is the origin $0$.
\end{enumalph}
\end{proposition}

\begin{proof}$~$

(a) The first part follows immediately from the surjectivity of $m$ in \eqref{moment-sl2} and the fact that $G\times^B\fraku^r$ is irreducible. Now since $G\times^B\fraku^r$ is a vector bundle over the base $G/B$ with each fiber isomorphic to $\fraku^r$, we have 
\[ \dim G\times^B\fraku^r=\dim G/B+\dim\fraku^r=r+1.\] 
As 
$m$ is birational, $C_r(\N)$ has the same dimension as $G\times^B\fraku^r$.

(b) By Proposition \ref{trivial singular} we immediately get that $0$ is a singular point of the variety. It is enough to show that every non-zero element in $C_r(\N)$ belongs to a smooth open subset of dimension $r+1$. Let $0\neq v=(v_1,\ldots,v_r)\in C_r(\N)$. We can assume that $v_1\ne 0\in\N$. Then considering the projection on the first factor $p:C_r(\N)\to\N$, we see that $v\in G\cdot(v_1,\fraku^{r-1})=p^{-1}(G\cdot v_1)$ which is open in $C_r(\N)$ as $G\cdot v_1$ is the regular orbit in $\N$. Now we define an action of the reductive group $G\times k^{r-1}$ on $C_r(\N)$ as follows:
\begin{align*}
\left(G\times k^{r-1}\right)\times C_r(\N) &\rightarrow C_r(\N)\\
(g\ ,\ a_1\ ,\ldots,\ a_{r-1})\bullet (v_1\ ,\ldots,\ v_r) &\longmapsto g\cdot(v_1\ ,\ a_1v_2\ ,\ldots,\ a_{r-1}v_r).
\end{align*}
It is easy to see that $p^{-1}(G\cdot v_1)=G\times k^{r-1}\bullet (v_1,v_1,\ldots,v_1)$. As every orbit is itself a smooth variety, we obtain $p^{-1}(G\cdot v_1)$ is smooth of dimension $r+1$.
\end{proof}

\subsection{Cohen-Macaulayness}
 
We denote by $\cap^*$ the scheme-theoretic intersection in order to distinguish with the regular intersection of varieties. Before showing that $C_r(\N)$ is Cohen-Macaulay, we need some lemmas related to Cohen-Macaulay varieties. The first one is an exercise in \cite[Exercise 18.13]{E:1995} (see also \cite[Lemma 5.15]{BV:1988}).
 
\begin{lemma}\label{Eisenbud}\footnote{In \cite{E:1995}, although the author did not state the word ``scheme-theoretic intersection", we implicitly understand from the exercise that the intersection needs to be scheme-theoretic.}
Let $X,Y$ be two Cohen-Macaulay varieties of the same dimension. Suppose the scheme-theoretic intersection $X\cap^* Y$ is of codimension 1 in both $X$ and $Y$. Then $X\cap^* Y$ is Cohen-Macaulay if and only if $X\cup Y$ is Cohen-Macaulay.
\end{lemma}

The following lemmas involve properties about radical ideals that determine whether a scheme-theoretic intersection $\cap^*$ coincides with the intersection of varieties $\cap$.

\begin{lemma}\label{lem of radical ideal}
Let $I\triangleleft k[x_1,\ldots,x_m]$ be the radical ideal associated to a variety $V$ in $\A^m$. Then the variety $V\times 0\subset\A^m\times\A^n$ is represented by the ideal $I+\left<y_1,\ldots,y_n\right>\triangleleft k[x_1,\ldots,x_r,y_1,\ldots,y_n]$.
\end{lemma}

\begin{proof}
It is easy to see that we have an isomorphism of rings
\[ \frac{k[x_1,\ldots,x_r,y_1,\ldots,y_n]}{I+\left<y_1,\ldots,y_n\right>}\cong\frac{k[x_1,\ldots,x_r]}{I} \]
where the latter ring is reduced. The result immediately follows.
\end{proof}

\begin{lemma}\label{radical-properties}
Let $I_1,I_2$ be radical ideals of $k[x_1,\ldots,x_m]$. Let $J_1,J_2$ be ideals of $R=k[x_1,\ldots,x_m,y_1,\ldots,y_n]$. If $I_2\subseteq I_1$ and $J_1\subseteq J_2$, then we have
\[ \sqrt{I_1+J_1}+\sqrt{I_2+J_2}=I_1+J_2 \]
provided $I_1+J_2$ is a radical ideal of $R$.

In particular, suppose $V_1\subseteq V_2$ are varieties of $\A^m$ and $W$ is a variety of $\A^n$. Then we have
\[ \left(V_1\times W\right) \cap^* \left( V_2\times 0\right) =\left( V_1\times W\right) \cap\left( V_2\times 0\right)= V_1\times 0. \]
\end{lemma}

\begin{proof}
It is well-known that $\sqrt{I_1+J_1}+\sqrt{I_2+J_2}\subseteq \sqrt{I_1+J_2}=I_1+J_2$. On the other hand, $I_1+J_2\subseteq \sqrt{I_1+J_1}+\sqrt{I_2+J_2}$. This shows the first part of the lemma.

For the remainder, let $I(V_1\times W)=\sqrt{I_1+J_1}$ with $I_1=I(V_1)$ and $J_1$ is an ideal of $R$. Let $I(V_2\times 0)=\sqrt{I_2+J_2}$ where $I_2=I(V_2), J_2=\left<y_1,\ldots,y_n\right>$ by the preceding lemma. Then we have $I_1+J_2$ is radical again by the previous lemma. It can be seen that $J_1\subseteq J_2$ as an ideal of $R$. Hence the first statement of the lemma implies that
\[  V_1\times W\cap^* V_2\times 0=\spec\frac{R}{I_1+J_2}=V_1\times 0. \]
\end{proof}

For each $r\ge 1$, the variety $C_r(\N)$ is defined as the zero locus of the family of polynomials  
\[ \{x_i^2+y_iz_i \ ,\ x_iy_j-x_jy_i\ ,\ x_iz_j-x_jz_i\ ,\ y_iz_j-y_jz_i~|~1\le i\le j\le r\}\]
in $R=k[\fraksl_2^r]=k[x_i,y_i,z_i~|~1\le i\le r]$. It is easy to check by computer that this ideal is not radical. This causes some difficulties for us to investigate the algebraic properties of this variety like Cohen-Macaulayness. Let $I_r$ be the radical ideal generated by the family of polynomials above in $R$. We first reduce the problem to checking a certain condition in commutative algebra.
  
\begin{lemma}\label{my lemma}
If $I_s+\left<y_1+z_1\right>$ is a radical ideal for each $s\ge 1$, then $C_r(\N)$ is Cohen-Macaulay for each $r\ge 1$. 
\end{lemma}

\begin{proof}
We argue by induction on $r$. When $r=1$, $C_1(\N)=\N$, which is a well-known Cohen-Macaulay variety. Suppose that $C_{r-1}(\N)$ is Cohen-Macaulay for some $r\ge 2$. As we have seen earlier $C_r(\N)$ is irreducible of dimension $r+1$ and $(y_1+z_1)$ is not zero in the coordinate ring $\frac{R}{I_r}$. The hypothesis and Lemma 5.15 in \cite{BV:1988} imply that it suffices to show the variety $C_r(\N)\cap V(y_1+z_1)$ is Cohen-Macaulay of dimension $r$.

Let $V=C_r(\N)\cap V(y_1+z_1)$. Solving from the constraint $x_1^2+y_1z_1=0$, we have either $x_1=y_1=-z_1$ or $x_1=-y_1=z_1$. Then we decompose $V=V_1\cup V_2\cup V_3$ where the $V_i$ are irreducible algebraic varieties defined by
\begin{align*}
V_1 &= V\cap V(x_1=y_1=z_1=0), \\
V_2 &= \overline{V\cap V(x_1=y_1=-z_1\ne 0)},\\
V_3 &= \overline{V\cap V(x_1=-y_1=z_1\ne 0)}.
\end{align*}
Moreover, we can explicitly describe these varieties as follows: 
\begin{align*}
V_1 &= 0\times 0\times 0\times C_{r-1}(\N)\\
V_2 &= \overline{\{(x_1\ ,\ y_1\ ,\ z_1\ ,\ x_2,\ \frac{x_2}{x_1}y_1\ ,\ \frac{x_2}{x_1}z_1\ ,\ldots,\ x_{r}\ ,\ \frac{x_{r}}{x_1}y_1\ , \frac{x_{r}}{x_1}z_1)~|~0\ne x_1=y_1=-z_1\in k\}} \\
&= \{(x_1\ ,\ x_1\ ,\ -x_1\ ,\ x_2\ ,\ x_2\ ,\ -x_2\ ,\ldots,\ x_{r}\ ,\ x_{r}\ ,\ -x_{r})~|~ x_i\in k\}, \\
V_3 &= \overline{\{(x_1\ ,\ y_1\ ,\ z_1\ ,\ x_2\ ,\frac{x_2}{x_1}y_1\ ,\frac{x_2}{x_1}z_1\ ,\ldots,\ x_{r+1}\ ,\frac{x_{r}}{x_1}y_1\ , \frac{x_{r}}{x_1}z_1)~|~0\ne x_1=-y_1=z_1\in k\}}\\
&= \{(x_1\ ,\ -x_1\ ,\ x_1\ ,\ x_2\ ,\ -x_2\ ,\ x_2\ ,\ldots,\ x_{r}\ ,\ -x_{r}\ ,\ x_{r})~|~ x_i\in k\}.
\end{align*}

Observe that $V_1$ is Cohen-Macaulay of dimension $r$ by the inductive hypothesis, and $V_2$ and $V_3$ are affine $r$-spaces so they are Cohen-Macaulay. From Lemma \ref{radical-properties}, we have
the scheme-theoretic intersection
\[V_1\cap^* V_2=\{(0 , 0 , 0 , x_2 , x_2 , -x_2 ,\ldots, x_{r} , x_{r} , -x_{r})~|~ x_i\in k,~2\le i\le r\}\]
which is an affine $(r-1)$-space. Hence by Lemma \ref{Eisenbud} the union $V_1\cup V_2$ is a Cohen-Macaulay variety of dimension $r$. Next we consider the scheme-theoretic intersection $(V_1\cup V_2)\cap^* V_3$. Note that $V_2\cap^* V_3=\{0\}$, we have
\[ (V_1\cup V_2)\cap^* V_3 = V_1\cap^* V_3 =\{(0 , 0 , 0 , x_2 , -x_2 , x_2 ,\ldots, x_{r} , -x_{r} , x_{r})~|~ x_i\in k\} \]
for the same reason as earlier. Then again Lemma \ref{Eisenbud} implies that $V_1\cup V_2\cup V_3$ is Cohen-Macaulay of dimension $r$.
\end{proof}

We are now interested in the conditions under which the sum of two radical ideals is again radical. One of the well-known concepts in commutative algebra related to this problem is that of a principal radical system introduced by Hochster, which shows certain class of ideals are radical.

\begin{theorem}\cite[Theorem 12.1]{BV:1988}\label{Hochster}
Let $A$ be a Noetherian ring, and let $F$ be a family of ideals in $A$, partially ordered by inclusion. Suppose that for every member $I \in F$ one of the following assumptions is fulfilled:
\begin{enumalph}
\item $I$ is a radical ideal; or
\item There exists an element $x \in A$ such that $I + Ax \in F$ and
\begin{itemize}
\item $x$ is not a zero-divisor $\frac{A}{\sqrt{I}}$ and $\bigcap_{i=0}^\infty (I + Ax^i)/I = 0$, or
\item there exists an ideal $J \in F$ with $I\subsetneq J$, such that $xJ \subseteq I$ and $x$ is not a zero-divisor of $\frac{A}{\sqrt{J}}$.
\end{itemize}
\end{enumalph}
Then all the ideals in $F$ are radical ideals.
\end{theorem}

Such a family of ideals is called a principal radical system. This concept plays an important role in the proof of Hochster and Eagon showing that determinantal rings are Cohen-Macaulay \cite{HE:1971}. Before applying this theorem, we need to set up some notations.

Fix $r\ge 1$. For each $1\le m\le r$, let $I_m$ be the radical ideal associated to the variety $0\times\cdots\times 0\times C_{m}(\N)\subseteq C_r(\N)$ with $0\in\N$. Each ideal $I_m$ is prime by the irreducibility of $C_m(\N)$ and it is easy to see that 
\[ I_m=I_r+\sum_{j=1}^{r-m}\left<x_j , y_j , z_j\right>. \] 
We also let, for each $1\le m\le r$, 
\[ P_m=\sum_{i=1}^m\left<x_i , y_i , z_i\right>+\sum_{j=m+1}^r\left<x_j-y_j\ ,\  y_j+z_j\right>. \]
Note that each $P_m$ is a prime ideal since $R/P_m$ is isomorphic to $k[x_{m+1},\ldots,x_r]$. Now we consider the following family of ideals in $R$
\begin{align*}
\calF= & \{I_j\}_{j=1}^r ~\cup ~ \{P_j\}_{j=1}^r ~\cup ~  \{I_r+\sum_{i=1}^m\left<y_i+z_i\right>\}_{m=1}^r  \\
& \cup \{I_r+\sum_{i=1}^r\left<y_i+z_i\right>+\sum_{i=1}^n\left<x_j+y_j\right>\}_{n=1}^r ~\cup ~ \mathfrak{m}=\left<x_1 , y_1 , z_1 ,\ldots, x_r , y_r , z_r\right>.
\end{align*}

\begin{proposition}\label{radical system}
The family $\calF$ is a principal radical system.
\end{proposition}

\begin{proof}
It is obvious that $ \{I_j\}_{j=1}^r $, $ \{P_j\}_{j=1}^r$, and $\mathfrak{m}$ are radical. So we just have to consider the two following cases:
\begin{enumalph}
\item $I=I_r+\sum_{i=1}^m\left<y_i+z_i\right>$ for some $1\le m\le r-1$. Observe that $I+\left<y_{m+1}+z_{m+1}\right>$ is an element in $\calF$. Let $J=I_{r-m}$. It is easy to see  $y_{m+1}+z_{m+1}\notin I_{r-m}$ so that $y_{m+1}+z_{m+1}$ is not a zero-divisor in the domain $R/\sqrt{J}=R/I_{r-m}$. It remains to show that $(y_{m+1}+z_{m+1})J\subseteq I$. Recall that $I_{r-m}=I_r+\sum_{j=1}^{m}\left<x_j,y_j,z_j\right>$. Then it suffices to prove that
\begin{align*}
\left(y_{m+1}+z_{m+1}\right)x_j \in I, \\
\left(y_{m+1}+z_{m+1}\right)y_j \in I, \\
\left(y_{m+1}+z_{m+1}\right)z_j \in I
\end{align*}
for all $1\le j\le m$. This is done in the Appendix \ref{checking 1}.

\item $I=I_r+\sum_{i=1}^r\left<y_i+z_i\right>+\sum_{i=1}^n\left<x_j+y_j\right>$ for some $0\le n\le r-1$, where if $n=0$, we set $I=I_r+\sum_{i=1}^r\left<y_i+z_i\right>$. It is clear that $I+\left<x_{n+1}+y_{n+1}\right>\in\calF$. Choose $J=P_{n}$, then the same argument as in the previous case gives us $x_{n+1}+y_{n+1}$ is not a zero-divisor of $R/\sqrt{J}$. We also refer the reader to the Appendix \ref{checking 2} for the proof of $(x_{n+1}+y_{n+1})J\subseteq I$. 
\end{enumalph} 
\end{proof}

Here is the main result of this section.

\begin{theorem}\label{Cohen-Macaulay of C_r(N(sl_2))}
For each $r\ge 1$, the variety $C_r(\N)$ is Cohen-Macaulay and therefore normal.
\end{theorem}
\begin{proof}
It immediately follows from Lemma \ref{my lemma} and Proposition \ref{radical system}.
\end{proof}

Now we can summarize our results into a theorem as we stated in Section \ref{main results}.

\begin{theorem}
Let $\mathfrak{gl}_2$ and $\fraksl_2$ be Lie algebras defined over $k$ of characteristic $p\ne 2$. Then for each $r\ge 1$, we have the commuting varieties $C_r(\mathfrak{gl}_2), C_r(\fraksl_2),$ and $C_r(\N)$ are irreducible, normal, and Cohen-Macaulay. 
\end{theorem}

\begin{remark}
We claim that the theorem above holds even when the field $k$ is not algebraically closed. Indeed, all the proofs in the last section do not depend on the algebraic closedness of $k$ since the theory of determinantal rings does not require it. There should be a way to avoid Nullstellensatz's Theorem (hence algebraic closedness is not necessary) in the arguments in this section. 
\end{remark}

\subsection{Rational singularities}
We prove in this section that the moment map $$m:G\times^B\fraku^r\rightarrow C_r(\N)$$ admits rational singularities. Since $m$ is already a resolution of singularities, it is equivalent to show the following.

\begin{proposition}\label{rational resolution}
$~$
\begin{enumalph}
\item $\calO_{C_r(\N)}=m_*\calO_{G\times^B\fraku^r}$.
\item The higher direct image $R^im_*(\calO_{G\times^B\fraku^r})=0$ for $i>0$. Hence $G\times^B\fraku^r$ is a rational resolution of $C_r(\N)$ via $m$.  
\end{enumalph}
\end{proposition}

\begin{proof}
\begin{enumalph}
\item This follows from Theorem \ref{Cohen-Macaulay of C_r(N(sl_2))}, and Zariski's Main Theorem \ref{Zariski}.

\item By \cite[Proposition 8.5]{H:1977}, we have $R^im_*(\calO_{G\times^B\fraku^r})\cong\calL\left(\opH^i(G\times^B\fraku^r,\calO_{G\times^B\fraku^r})\right)$ for each $i\ge 0$ (as we pointed out in the end of Section \ref{algebraic geometry conventions}). Note that $\calL(k)=\calO_{G\times^B\fraku^r}$, so we have
\[ \opH^i(G\times^B\fraku^r,\calO_{G\times^B\fraku^r})\cong\bigoplus_{j=0}^\infty\opH^i(G/B,\calL_{G/B}S^j(\fraku^{r*}))=\bigoplus_{j=0}^\infty R^i\ind_B^G(S^j(\fraku^{*r})). \]
As we are assuming that $\fraku$ is a one-dimensional space of weight corresponding to the negative root $-\alpha$, $S^j(\fraku^{*r})$ can be considered as the direct sum $\left(k_{j\alpha}\right)^{\oplus P_r(j)}$, where $P_r(j)$ is the number of ways getting $r$ as a sum of $j$ non-negative numbers. This weight is dominant, so by Kempf's vanishing theorem we obtain $R^i\ind_B^G(S^j(\fraku^{*r}))=0$ for all $i>0$ and $j\ge 0$.
It follows that
\[ R^im_*(\calO_{G\times^B\fraku^r})=0 \]
for all $i\ge 1$.
\end{enumalph}
\end{proof}
Now we state the result of this section.
\begin{theorem}
The moment map $m:G\times^B\fraku^r\rightarrow C_r(\N)$ is a rational resolution of singularities.
\end{theorem}

\subsection{Applications}\label{characters of coordinate alg of comm var for sl_2}
As a corollary, we establish the connection with the reduced cohomology ring of $G_r$.  This gives an alternative proof for the main result in \cite{BFS1:1997} for the special case $G=SL_2$.

\begin{theorem}\label{coordinate-alg}
For each $r>0$, there is a $G$-equivariant isomorphism of algebras 
$$k[G\times^B\fraku^r]\cong k[C_r(\N)].$$
Consequently, there is a homeomorphism between ${\rm{Spec}}\opH^\bullet(G_r,k)_{\red}$ and $C_r(\N)$.
\end{theorem}
\begin{proof}
Note that the moment map $m$ is $G$-equivariant. This implies that the comorphism $m^*$ is compatible with the $G$-action. The isomorphism follows from part (a) of Proposition \ref{rational resolution}. Combining this observation with \cite[Proposition 5.2.4]{N:2012}, we have the homeomorphism between the spectrum of the reduced $G_r$-cohomology ring and $C_r(\N)$. 
\end{proof}

Now we want to describe the characters for the coordinate algebra $k[C_r(\N)]$. Before doing that, we need to introduce some notation. Here we follow the convention in \cite{Jan:2004}.

Let $V$ be a graded vector space over $k$, i.e., $V=\bigoplus_{i\ge 0}V_n$, where each $V_n$ is finite dimensional. Set
\[ \ch_tV=\sum_{n\ge 0}\ch(V_n) t^n \]
the character series of $V$.

From the isomorphism in Theorem \ref{coordinate-alg}, the character series for the coordinate algebra of $C_r(\N)$ can be computed via that of $k[G\times^B\fraku^r]$: 
\begin{equation}\label{character}
\ch_t\left(k[C_r(\N)]\right)=\ch_t\left(k[G\times^B\fraku^r]\right).
\end{equation}

\begin{theorem}
For each $r>0$, we have 
\begin{equation}\label{char}
\ch_tk[C_r(\N)]=\sum_{n\ge 0}\sum_{a_1+\cdots+a_r=n}\chi(n\alpha)t^n.
\end{equation}
where the latter sum is taken over all $r$-tuple $(a_1,\ldots,a_r)$ of non-negative integers satisfying $$a_1+\cdots+a_r=n.$$
\end{theorem}

\begin{proof}
From \cite[8.11(4)]{Jan:2004}, $k[G\times^B\fraku^r]$ is a graded $G$-algebra and for each $n\ge 0$,
\[ \ch k[G\times^B\fraku^r]_n = \ch\opH^0(G/B,S^n(\fraku^{r*})). \]
The argument in Proposition \ref{rational resolution} gives us  
\[ \opH^i(G/B,S^n(\fraku^{*r}))=0\]
for all $i>0,n\ge 0$. Hence, it follows by \cite[8.14(6)]{Jan:2004} that
\[ \chi(S^n(\fraku^{*r})) = \ch\opH^0(G/B,S^n(\fraku^{*r})) \]
for each $n\ge 0$. Here we recall that for each finite-dimensional $B$-module $M$, the Euler characteristic of $M$ is defined as
\[ \chi(M)= \sum_{i\ge 0}(-1)^i\ch\opH^i(G/B, M). \]
Then we write $\ch(S^n(\fraku^*))=e(n\alpha)$; hence for each $n$ we have
\begin{align*}
\ch(S^n(\fraku^{r*})) &=\ch(S^n(\fraku^*)^{\otimes r})\\
&=\ch\left(\sum_{a_1+\cdots+a_r=n}S^{a_1}(\fraku^*)\otimes\cdots\otimes S^{a_r}(\fraku^*)\right)\\
&=\sum_{a_1+\cdots+a_r=n}\ch\left(S^{a_1}(\fraku^*)\otimes\cdots\otimes S^{a_r}(\fraku^*)\right)\\
&=\sum_{a_1+\cdots+a_r=n} e(n\alpha).
\end{align*}
Thus by the linear property of $\chi$, we obtain
\begin{align*}
\chi\left(S^n(\fraku^{*r})\right) =\sum_{a_1+\cdots+a_r=n}\chi( e(n\alpha))=\sum_{a_1+\cdots+a_r=n}\chi(n\alpha).
\end{align*}
Combining all the formulas, we get
\[ \ch_tk[G\times^B\fraku^r] =\sum_{n\ge 0}\ch\left(k[G\times^B\fraku^r]_n\right)t^n =\sum_{n\ge 0}\sum_{a_1+\cdots+a_r=n}\chi(n\alpha)t^n. \]
Therefore, the identity \eqref{character} completes our proof.
\end{proof}

\begin{remark}
The above formula can be made more specific by making use of partition functions. Let $S$ be the set of $r$ numbers $a_1,\ldots,a_r\in\mathbb{N}$. Set $P_S:\mathbb{N}\rightarrow\mathbb{N}$ with $P_S(m)$ the number of ways getting $m$ as a linear combination of elements in $S$ with non negative coefficients. This is called a {\it partition function}. Kostant used this notation with $S=\Pi=\{\alpha_1,\ldots,\alpha_n\}$, the set of simple roots, to express the number of ways of writing a weight $\lambda$ as a linear combination of simple roots, see \cite{SB:1984} for more details. In our case, set $S=\{\alpha_1=\cdots=\alpha_r=1\}$ and denote $P_r=P_S$. Then from \eqref{char} we can have
\[ \ch k[C_r(\N)]_n=P_r(n)\chi(n\alpha).\]
\end{remark}
As a consequence, we obtain a result on the multiplicity of $\opH^0(\lambda)$ in $k[C_r(\N)]$.
\begin{corollary}
For each $\lambda=m\alpha\in X^+$, we have 
\[ \left[k[C_r(\N)]:\opH^0(\lambda)\right]=P_r(m).\]
\end{corollary}
This result also shows that the coordinate algebra of $C_r(\N)$ has a good filtration as a $G$-module.

\section{The Mixed Cases}\label{mixed cases}
\subsection{} Now we turn our concerns into more complicated commuting varieties. Let $V_1,\ldots, V_r$ be closed subvarieties of a Lie algebra $\g$. Define
\[ C(V_1,\ldots, V_r)=\{(v_1,\ldots, v_r)\in V_1\times\cdots\times V_r~|~[v_i,v_j]=0~,~1\le i\le j\le r \}, \]   
a mixed commuting variety over $V_1\times\cdots\times V_r$. It is obvious that when $V_1=\cdots=V_r$, this variety coincides with the commuting variety over $V_1$. We apply in this section our calculations from previous sections to study mixed commuting varieties over $\fraksl_2$ and its nilpotent cone $\N$.

We assume $k$ is an algebraically closed field of characteristic $p\ne 2$. Set
\[ C_{i,j}=C(\underbrace{\N,\ldots,\N}_{i~ \text{times}},\underbrace{\fraksl_2,\ldots,\fraksl_2}_{j~ \text{times}}) \]
with $i,j\ge 1$, a mixed commuting variety over $\fraksl_2$ and $\N$. Note that if $j=0$, then $C_{i,0}=C_i(\N)$. Otherwise, one gets $C_j(\fraksl_2)$ if $i=0$. By permuting the tuples, this variety is isomorphic to any mixed commuting variety in which $\N$ appears $i$ times and $\fraksl_2$ appears $j$ times. Hence we consider $C_{i,j}$ a representative of such varieties. The following are some basic properties.

\begin{proposition}\label{prop:C_i,j_decomposition}
For each $i,j\ge 1$, we have: 
\begin{enumalph}
\item The variety $C_{i,j}$ is reducible. Moreover, we have $C_{i,j}=C_{i+j}(\N)~\cup~ 0\times\ldots\times 0\times C_j(\fraksl_2)$.
\item $\dim C_{i,j}=i+j+1.$
\item $C_{i,j}$ is not normal. It is Cohen-Macaulay if and only if $i=1$. 
\end{enumalph}
\end{proposition}
\begin{proof}
Part (a) and (b) immediately follow from the decomposition
\[ C_{i,j}=C_{i+j}(\N)~\cup~ 0\times\ldots\times 0\times C_j(\fraksl_2). \]
Indeed, the inclusion $\supseteq$ is obvious. For the other inclusion, we consider $v=(v_1,\ldots,v_{i+j})\in C_{i,j}$. If $v_m\ne 0$ for some $1\le m\le i$, then it is distinguished. Hence by \cite[Corollary 35.2.7]{TY:2005} the centralizer of $v_m$ is in $\N$. This implies that all the $v_n$ with $i+1\le n\le i+j$ must be in $\N$ so that $v\in C_{i+j}(\N)$. Otherwise, we have $v_1=\cdots=v_i=0$; hence $v\in 0\times\ldots\times 0\times C_j(\fraksl_2)$. Also note from earlier that $\dim C_{i+j}(\N)=i+j+1$ and $\dim C_j(\fraksl_2)=j+2$.

(c) The mixed commuting variety $C_{i,j}$ is never normal for all $i,j\ge 1$ as it is reducible and the irreducible components overlap each other (they all share the origin). Recall that a variety is Cohen-Macaulay only if its irreducible components are equidimensional. In our case, this happens only when $i=1$. The remainder is to show that $C_{1,j}=C(\N,\fraksl_2,\ldots,\fraksl_2)$ is Cohen-Macaulay. Let 
\[ I(C_{j+1}(\N))=\sqrt{I(C_j(\N))+J_1} \]
where $J_1=\left< x_1^2+y_1z_1, x_1y_i-x_iy_1, x_1z_i-x_iz_1, y_1z_i-y_iz_1~ \mid~ i=2,\ldots,j+1\right>$, an ideal of $k[\fraksl_2^{j+1}]$. Note also by Lemma \ref{lem of radical ideal} that
\[ I(0\times C_j(\fraksl_2))=\left<x_1,y_1,z_1\right>+I(C_j(\fraksl_2)) \]
and that $I(C_j(\fraksl_2))\subseteq I(C_j(\N)), J_1\subseteq \left<x_1,y_1,z_1\right>$. Hence, by Lemma \ref{radical-properties}, we have
\[ C_{j+1}(\N)\cap^* 0\times C_j(\fraksl_2)=C_{j+1}(\N)\cap 0\times C_j(\fraksl_2)=0\times C_j(\N). \]
Observe, in addition, that each variety is Cohen-Macaulay. So Lemma \ref{Eisenbud} implies that $C_{1,j}$ is Cohen-Macaulay.
\end{proof}

\begin{remark}
One can see that mixed commuting varieties are the generalization of commuting varieties. Strategies in studying the latter objects can be applied to that of the first objects. However, we have not seen any work in the literature related to this concept. Our results in this section show that mixed commuting varieties are usually not irreducible, hence fail to be normal or Cohen-Macaulay. The reducibility also causes big dificulties in computing their dimensions. If $V_1$ and $V_n$ are the minimal and maximal element of the family $V_1,\ldots,V_n$ ordered by inclusion, then for each $r\ge 1$ one can easily see the following
\[ \dim C_r(V_1)\le \dim C(V_1,\ldots,V_n)\le \dim C_r(V_n). \]
These bounds are rough and depend on the dimensions of the commuting varieties. So new methods are required to investigate a mixed commuting variety over various closed sets in a higher rank Lie algebra.
\end{remark}

\section{Commuting Variety of 3 by 3 matrices}\label{C_r(N(sl_3))}\label{comvar of 3 by 3}
We turn our assumption to $G=SL_3$ and $\g=\fraksl_3$ defined over an algebraically closed field $k$. Recall that $\N$ denotes the nilpotent cone of $\g$. Not much is known about commuting varieties in this case. In the present section, we only focus on the calculations on the nilpotent commuting variety. In particular, we apply determinantal theory to study the smaller variety $C_r(\fraku)$, and then use the moment morphism to obtain results on $C_r(\N)$. We next show that for each $r\ge 1$ all of the singular points of $C_r(\N)$ are in the commuting variety $C_r(\overline{\calO_{\sub}})$.

\subsection{Irreducibility} First, we identify $\fraku^r$ with the affine space as follows.
\[ \fraku^r=\left\{ \left( \begin{array}{ccc} 0 & 0  & 0\\ x_1 & 0 & 0 \\ y_1 & z_1 & 0 \end{array} \right), \ldots , \left( \begin{array}{ccc} 0 & 0  & 0\\ x_r & 0 & 0 \\ y_r & z_r & 0 \end{array} \right)~|~x_i\ ,\ y_i\ ,\ z_i\in k, 1\le i\le r \right\}. \]
For each pair of elements in $\fraku$, the commutator equations are $x_iz_j-x_jz_i$ with $1\le i\le j\le r$. It follows that
\[ k[C_r(\fraku)]=\frac{k[x_1 , y_1 , z_1 ,\ldots, x_r , y_r , z_r]}{\sqrt{\left<x_iz_j-x_jz_i~|\ 1\le i\le j\le r\right>}}. \]
\begin{proposition}\label{commuting u(sl_3)}
For each $r\ge 1$, the variety $C_r(\fraku)$ is
\begin{enumalph}
\item irreducible of dimension $2r+1$,
\item singular at the point 
\[ X=\left( \left( \begin{array}{ccc} 0 & 0  & 0\\ 0 & 0 & 0 \\ y_1 & 0 & 0 \end{array} \right),\cdots,\left( \begin{array}{ccc} 0 & 0  & 0\\ 0 & 0 & 0 \\ y_r & 0 & 0 \end{array} \right) \right), \] 
with $y_1,\ldots,y_r\in k$,
\item Cohen-Macaulay and normal.
\end{enumalph}
\end{proposition}

\begin{proof}
\begin{enumalph}
We first consider the following isomorphism of rings
\begin{equation}\label{iso}
k[C_r(\fraku)] \cong k[y_1,\ldots,y_r]\otimes\frac{k[x_1,z_1,\ldots,x_r,z_r]}{\sqrt{\left<x_iz_j-x_jz_i~|\ 1\le i\le j\le r\right>}}.
\end{equation}
Let $V$ be the variety associated to the second ring in the tensor product. Then parts (a) and (c) follow if we are able to show that $V$ is irreducible, Cohen-Macaulay and normal. Indeed, these properties can be obtained from determinantal varieties as we argued earlier. 

Consider the matrix
\[\mathcal{X}=\left( \begin{array}{cccc} x_1 & x_2 & \cdots & x_r\\ z_1 & z_2 & \cdots & z_r \end{array} \right). \]
It is easy to see that 
\[ k[V]\cong \frac{k(\mathcal{X})}{I_2(\mathcal{X})}\]
where $I_2(\mathcal{X})$ is the ideal generated by 2-minors of $\mathcal{X}$. Hence by Theorem \ref{determinantal rings}, $k[V]$ is a Cohen-Macaulay, normal domain of Krull dimension $r+1$. As tensoring with a polynomial ring does not change the properties except the dimension, we have $k[C_r(\fraku)]$ is a Cohen-Macaulay, normal domain of Krull dimension $2r+1$ \cite{BK:2002}. This proves parts (a) and (c).

(b) Note that the group $H=GL_r\times GL_2$ acts on $V$ under the following action:
\[ \left( \left( \begin{array}{ccc} a_{11} & \cdots & a_{1r} \\
\vdots & \ddots & \vdots \\
a_{r1} & \cdots & a_{rr} \end{array} \right), \left( \begin{array}{cc} a &  b \\
c & d \end{array} \right)\right) \times\left[ \left( \begin{array}{ccc} 0 & 0 & 0 \\
x_i & 0 & 0 \\
0 & z_i & 0 \end{array} \right)\right]_{i=1}^r \mapsto \left( \sum_{i=1}^ra_{1i}v_i\ ,\ldots,\ \sum_{i=1}^ra_{1i}v_j \right) \]
where $v_i=\left( \begin{array}{ccc} 0 & 0 & 0 \\
ax_i+bz_i & 0 & 0 \\
0 & cx_i+dz_i & 0 \end{array} \right)$ for each $1\le i\le r$. Now consider a nonzero element $v$ in $V$. Without loss of generality, we can assume that the entry $x_1\ne 0$. Then $v$ belongs to the open set $W=V\cap V(x_1\ne 0)$. On the other hand, it is not hard to see that 
\[ W=H\cdot \left[ \left( \begin{array}{ccc} 0 & 0 & 0 \\
1 & 0 & 0 \\
0 & 0 & 0 \end{array} \right),\ldots ,\left( \begin{array}{ccc} 0 & 0 & 0 \\
1 & 0 & 0 \\
0 & 0 & 0 \end{array} \right)\right] \]
which is a smooth orbit of dimension $r+1$. It follows that $v$ is non-singular. From the isomorphism \eqref{iso}, we have the set of all singular points of the right-hand side is $V(k[y_1,\ldots,y_r])\times\{0\}$. This determines the singular locus as desired.
\end{enumalph}
\end{proof}
Now we use properties of $C_r(\fraku)$ as ingredients to study the irreducibility and dimension of $C_r(\N)$.
\begin{theorem}
The variety $C_r(\N)$ is an irreducible variety of dimension $2r+4$. 
\end{theorem}
\begin{proof}
Irreducibility follows from Theorem \ref{C(u)-and-C(N)}. Then, the birational properness of the moment morphism 
\[ m:G\times^BC_r(\fraku)\to C_r(\N)\]
implies that 
\[ \dim C_r(\N)=\dim G\times^BC_r(\fraku)=\dim C_r(\fraku)+\dim G/B=2r+4. \]
\end{proof}

\subsection{Singularities}
Serre's Criterion states that a variety $V$ is normal if and only if the set of singularities has codimension $\ge 2$ and the depth of $V$ at every point is $\ge\min(2,\dim V)$. This makes the task of determining the dimension of the singular locus of a variety necessary in order to verify normality. 

Note that the problem on the singular locus of ordinary commuting varieties was studied by Popov \cite{Po:2008}. He showed that the codimension of singularities for $C_2(\g)$ is greater than or equal to 2 for an arbitrary reductive Lie algebra $\g$. In other words, the variety $C_2(\g)$ holds the necessary condition for being normal. This author has not seen any analogous work in the literature for arbitrary commuting varieties. 

We prove in this subsection that the set of all singularities for $C_r(\N)$ has codimension $\ge 2$. Let $\alpha,\beta$ be the simple roots of the underlying root system $\Phi$ of $G$. Then we have the set of positive roots $\Phi^+=\{\alpha,\beta,\alpha+\beta\}$. Recall that $z_{\reg}$ and $z_{\sub}$ are respectively the centralizers of $v_{\reg}$ and $v_{\sub}$ in $\g$. We have shown in Proposition \ref{commuting u(sl_3)}(b) that the vector space $\fraku_{-\alpha-\beta}^r$ includes all singularities of $C_r(\fraku)$. Before locating the singularities of $C_r(\N)$, we need some lemmas.

\begin{lemma}\label{lemma-of-singularities}
Suppose $r\ge 2$. For each  $1\le i\le r$, the subset $V[i]=G\cdot(z_{\reg}\ ,\ldots,\ v_{\reg}\ ,\ldots,\ z_{\reg})$ with $v_{\reg}$ in the $i$-th of $C_r(\N)$ is a smooth open subvariety of dimension $2r+4$.
\end{lemma}

\begin{proof}
As a linear combination of commuting nilpotents is also nilpotent, we define an action of the algebraic group $G\times\mathcal{M}_{r-1,r}$ on $C_r(\N)$ where $\mathcal{M}_{r-1,r}$ is the set of all $r-1\times r$-matrices and can be identified with $(\mathbb{G}^r)^{r-1}$:
\begin{align*}
\left(G\times \mathcal{M}_{r-1,r}\right)\times C_r(\N) &\rightarrow C_r(\N) \\
(g,(a_{ij}))\bullet(v_1,\ldots,v_r) &\longmapsto g\cdot(\sum_{j=1}^ra_{1j}v_j\ ,\ldots, v_i\ ,\ldots,\ \sum_{j=1}^ra_{r-1,j}v_{j}).
\end{align*}
for every element $(v_1,\ldots,v_r)\in C_r(\N)$. Now observe that $z_{\reg}$ is a vector space of dimension 2, choose $\{v_{\reg},w\}$ a basis of this space. It is easy to check that for $r\ge 2$ we have
\[ V[i]=(G\times\mathcal{M}_{r-1,r})\bullet(w\ ,\ 0\ ,\ldots,\ v_{\reg}\ ,\ldots,\ 0). \]
In other words, our original variety is an orbit under the bullet action. So 
it is smooth. 

Now for each $1\le i\le r$, consider the projection map from $C_r(\N)$ to the $i$-th factor, $p:C_r(\N)\to \N$. We then have
\[ V[i]=p^{-1}(\calO_{\reg})\]
which is an open subset in $C_r(\N)$. As the variety $C_r(\N)$ is irreducible, the dimension of $V[i]$ is the same as $\dim C_r(\N)$.
\end{proof}

\begin{lemma}\label{lemma-of-O_subreg}
The intersection of $z_{\sub}$ and $\overline{\calO_{\sub}}$ is exactly the union of $\fraku_{-\alpha}\times\fraku_{-\alpha-\beta}$ and $\fraku_{-\alpha}\times\fraku_{\beta}$. Hence, we have $\dim\ (z_{\sub}\cap\overline{\calO_{\sub}})=2$.
\end{lemma}

\begin{proof}
It can be computed that $z_{\sub}$ consists of all matrices of the form
\[ \left( \begin{array}{ccc} x_1 & 0 & 0 \\ x_2 & x_1 & t_2 \\ x_3 & 0 & -2x_1 \end{array} \right) \]
where $x_1\ ,\ x_2\ ,\ x_3\ ,\ t_2$ are in $k$. As the determinant of a nilpotent matrix is always $0$, we obtain
\[ z_{\sub}\cap\N=\left\{\left( \begin{array}{ccc} 0 & 0 & 0 \\ x_2 & 0 & t_2 \\ x_3 & 0 & 0 \end{array} \right)~|~x_2,x_3,t_2\in k \right\}. \]
On the other hand, it is well-known that $\calO_{\sub}$ consists of all matrices of rank one and $\overline{\calO_{\sub}}=\calO_{\sub}\cup\{0\}$. Then
\begin{align*}
z_{\sub}\cap\overline{\calO_{\sub}} &= \left\{ \left(\begin{array}{ccc} 0 & 0 & 0 \\ x_2 & 0 & 0 \\ x_3 & 0 & 0 \end{array} \right)\right\}\cup \left\{ \left(\begin{array}{ccc} 0 & 0 & 0 \\ x_2 & 0 & t_2 \\ 0 & 0 & 0 \end{array} \right)\right\} \\
&=\fraku_{-\alpha}\times\fraku_{-\alpha-\beta}\cup\fraku_{-\alpha}\times\fraku_{\beta}.
\end{align*}
It immediately follows that $\dim\ (z_{\sub}\cap\overline{\calO_{\sub}})=2.$
\end{proof}

For each $r\ge 1$, let $C^{\rm{sing}}_r$ be the singular locus of $C_r(\N)$. The following theorem determines the location of $C_r^{\rm{sing}}$.

\begin{theorem}\label{codim of sing}
For each $r\ge 1$, we have $C_r^{\rm{sing}}\subseteq C_r(\overline{\calO_{\sub}})$. Moreover, $\codim C_r^{\rm{sing}}\ge 2$.
\end{theorem}

\begin{proof}
It is obvious that our result is true for $r=1$. Assume now that $r\ge 2$. It suffices to prove that any element in the complement of $C_r(\overline{\calO_{\sub}})$ in $C_r(\N)$ is smooth. Let $V=C_r(\overline{\calO_{\sub}})$ and suppose $w\in C_r(\N)\backslash V$. Say $w=(w_1,\ldots,w_r)$ with some $w_{n}\notin\overline{\calO_{\sub}}$, i.e., $w_{n}$ is a regular element. Consider the projection onto the $n$-th factor $p_n:C_r(\N)\rightarrow\N$. As $G\cdot w_{n}=\calO_{\reg}$ is open in $\N$, the preimage $p_n^{-1}(G\cdot w_{n})$ is an open set of $C_r(\N)$. Note that 
\[ p_n^{-1}(G\cdot w_{n})=G\cdot p_n^{-1}(w_{n})=G\cdot(z_{\reg}\ ,\ldots,\ v_{\reg}\ ,\ldots,\ z_{\reg}).\] 
By Lemma \ref{lemma-of-singularities}, $w$ is non-singular. Therefore, $V$ contains all the singularities of $C_r(\N)$.

We are now computing the dimension of $\dim C_r(\overline{\calO_{\sub}})$. It is observed for each $r\ge 1$ that
\[ C_r(\overline{\calO_{\sub}})=\overline{G\cdot(v_{\sub},C_{r-1}(z_{\sub}\cap\overline{\calO_{\sub}}))}. \]
Now let $V_1=\fraku_{\alpha}\times\fraku_{\alpha+\beta}$ and $V_2=\fraku_{\alpha}\times\fraku_{-\beta}$. By the preceding lemma, we have $z_{\sub}\cap\overline{\calO_{\sub}}=V_1\cup V_2$. Also note that for $u, v\in V_1\cup V_2$ we have
\[ [u,v]=0\quad\Leftrightarrow\quad u,v\in V_1\quad\text{or}\quad u,v\in V_2.\]
It follows that $C_{r-1}(z_{\sub}\cap\overline{\calO_{\sub}})=V_1^{r-1}\cup V_2^{r-1}$. So we obtain 
\begin{equation*}\label{C_rO_sub}
C_r(\overline{\calO_{\sub}})=\overline{G\cdot(v_{\sub}\ ,V_1\ ,\ldots,\ V_1)}\cup\overline{G\cdot(v_{\sub}\ ,\ V_2\ ,\ldots,\ V_2)}.
\end{equation*}
This implies that $C_r(\overline{\calO_{\sub}})$ is reducible. Using theorem of dimension on fibers, we can further compute that $\overline{G\cdot(v_{\sub},\ V_j\ ,\ldots,\ V_j)}$ has dimension $2r+2$ for each $j=1,2$. Thus $\dim C_r(\overline{\calO_{\sub}})= 2r+2$ so that $\codim C_r(\overline{\calO_{\sub}})=2$. 
\end{proof}

\section*{Acknowledgments}

This paper is based on part of the author's Ph.D. thesis. The author gratefully acknowledges the guidance of his thesis advisor Daniel K. Nakano. Thanks to Christopher Drupieski for his useful comments. We also thank William Graham for discussions about Cohen-Macaulay rings. Finally, we are grateful to Alexander Premet and Robert Guralnick for the information about ordinary commuting varieties.

\section{Appendix}
We verify in this section the condition of Theorem \ref{Hochster}(b) for $I$ and $J$ are in the context of Proposition \ref{radical system}.

\subsection{Case 1}\label{checking 1} We check the following
\begin{align*}
(1)~~~ \left(y_{m+1}+z_{m+1}\right)x_j \in I, \\
(2)~~~ \left(y_{m+1}+z_{m+1}\right)y_j \in I, \\
(3)~~~ \left(y_{m+1}+z_{m+1}\right)z_j \in I.
\end{align*}
For the first one, we consider for each $1\le j\le m$,
\begin{align*}
(y_{m+1}+z_{m+1})x_j + I &=y_{m+1}x_j+z_{m+1}x_j + I \\
&=x_{m+1}y_j+x_{m+1}z_j + I \\
&= x_{m+1}\left<y_j+z_j\right> + I \\
&= I 
\end{align*}
where the second identity is provided by $x_{m+1}y_j-x_jy_{m+1}\ ,\ x_{m+1}z_j-x_jz_{m+1}\in I_r\subseteq I_m$; and the last identity is provided by $y_j+z_j\in I$. 

The same technique is applied for showing (2) and (3) as following: 
\begin{align*}
(y_{m+1}+z_{m+1})y_j + I &=y_{m+1}y_j+z_{m+1}y_j + I \\
&=y_{m+1}y_j+y_{m+1}z_j + I \\
&= y_{m+1}(y_j+z_j) + I \\
&= I, \\ 
(y_{m+1}+z_{m+1})z_j + I &=y_{m+1}z_j+z_{m+1}z_j + I \\
&=z_{m+1}y_j+z_{m+1}z_j + I \\
&= z_{m+1}(y_j+z_j) + I \\
&= I. 
\end{align*}

\subsection{Case 2}\label{checking 2} We need to check the following
\begin{align*}
(1) &~~~ \left(x_{n+1}+y_{n+1}\right)x_j \in I, \\
(2) &~~~ \left(x_{n+1}+y_{n+1}\right)y_j \in I, \\
(3) &~~~ \left(x_{n+1}+y_{n+1}\right)z_j \in I, \\
(4) &~~~\left(x_{n+1}+y_{n+1}\right)(x_{h}-y_{h}) \in I,
\end{align*}
for all $1\le j\le n$ and $n+1\le h\le r$. Verifying (1), (2), and (3) is similar to our work in Case 1. Lastly, we look at
\begin{align*}
(x_{n+1}+y_{n+1})(x_{h}-y_{h}) + I &= x_{n+1}x_h+y_{n+1}x_{h}-x_{n+1}y_{h} - y_{n+1}y_h + I \\
&= x_{n+1}x_h-y_{n+1}y_h + I \\
&= x_{n+1}x_h+y_{n+1}z_h + I. 
\end{align*}
Now in order to complete our verification, we will show that $x_{n+1}x_h+y_{n+1}z_h\in I_r$. Indeed, we have
\begin{align*}
(x_{n+1}x_h+y_{n+1}z_h)^2+I_r &=x_{n+1}^2x_h^2+2x_{n+1}x_hy_{n+1}z_h+y_{n+1}^2z_h^2 +I_r\\
&=-x_{n+1}^2y_hz_h+2x_{n+1}^2y_hz_h+y_{n+1}^2z_h^2 +I_r \\
&=x_{n+1}^2y_hz_h+y_{n+1}^2z_h^2 +I_r\\
&=-y_{n+1}z_{n+1}y_hz_h+y_{n+1}^2z_h^2 +I_r\\
&=y_{n+1}z_h(-z_{n+1}y_h+y_{n+1}z_h)+I_r \\
&=I_r.
\end{align*}
As $I_r$ is radical, we obtain $x_{n+1}x_h+y_{n+1}z_h\in I_r$ as desired.

\providecommand{\bysame}{\leavevmode\hbox to3em{\hrulefill}\thinspace}


\begin{thebibliography}{BNPP}


\bibitem[B]{Ba:2007}
R.~Basili, {\em On the irreducibility of commuting varieties of nilpotent matrices}, J. Pure Appl. Algebra, \textbf{149} (2000), 107--120.

\bibitem[Ba]{Ba:2001}
V.~Baranovsky, {\em The varieties of pairs of commuting nilpotent matrices is irreducible}, Transform. Groups, \textbf{6} (2001), 3--8.


\bibitem[BI]{BI:2008}
R.~Basili and A.~Iarrobino, {\em Pairs of commuting nilpotent matrices,
and Hilbert function}, J. Algebra, \textbf{320} (2008), 1235--1254.

\bibitem[BK]{BK:2002}
S.~ Bouchiba and S.~ Kabbaj, {\em Tensor products of Cohen-Macaulay rings: solution to a problem of Grothendieck}, J. Algebra, \textbf{252} (2002), 65--73.

\bibitem[BV]{BV:1988}
W.~Bruns and U.~Vetter, {\em Determinantal Rings}, Lectures Notes in Math., Springer-Verlag, 1988.

\bibitem[CM]{CM:1993}
D.~H. Collingwood and W.~M. McGovern, {\em Nilpotent Orbits in Semisimple Lie Algebras}, Van Nostrand Reinhold Mathematics Series, 1993.

\bibitem[E]{E:1995}

D.~Eisenbud, {\em Commutative Algebra with a view toward Algebraic Geometry}, Springer, 1995.

\bibitem[G]{G:1961}

M.~Gerstenhaber, {\em On dominance and varieties of commuting Matrices}, Ann. Math., \textbf{73} (1961), 324--348.


\bibitem[GS]{GS:2000}

R.~M. Guralnick and B.~A. Sethuraman,{\em Commuting pairs and triples of matrices and related varieties }, Linear Algebra Appl., \textbf{310} (2000), 139--148.

\bibitem[H]{H:1977}

R.~Hartshorne, {\em Algebraic Geometry}, Graduate Texts in Mathematics, Springer-Verlag, 1977.

\bibitem[HE]{HE:1971}
M.~Hochster and A.~J. Eagon, {\em Cohen-Macaulay rings, invariant theory, generic perfection of determinantal loci}, Amer. J. Math., \textbf{93} (1971), 1020--1058.

\bibitem[Hum1]{Hum:1978}
J.~E. Humphreys, {\em Introduction to {L}ie Algebras and Representation
  Theory}, Graduate Texts in Mathematics, Springer-Verlag, 1978.

\bibitem[Hum2]{Hum:1995}

\bysame, {\em Linear Algebraic Groups}, Graduate Texts in Mathematics, Springer-Verlag, 1995.

\bibitem[Hr]{H:2006}

F.~Hreinsdottir, {\em Miscellaneous results and conjectures on the ring of commuting matrices}, An. St. Univ. Ovidius Constanta, \textbf{14} (2006), 45--60.

\bibitem[Jan]{Jan:2004}

J.~C. Jantzen and K-H.~Neeb, {\em Lie Theory: Lie Algebras and Representations}, Progress in Mathematics, Birkhauser, \textbf{228} 2004.

\bibitem[K]{K:2005}

A.~Knutson, {\em Some schemes related to the commuting variety}, J. Algebraic Geom., \textbf{14} (2005), 283--294.

\bibitem[KN]{KN:1987}

A.~A. Kirillov and Y.~A. Neretin, {\em The variety $A_n$ of $n$-dimensional Lie algebra structures}, Amer. Math. Soc. Transl., \textbf{137} (1987), 21--30.

\bibitem[L]{Le:2002}
P.~Levy, {\em Commuting varieties of Lie algebras over fields of prime characteristic}, J. Algebra, \textbf{250} (2002), 473--484.

\bibitem[MS]{MS:2011}
M.~Majidi-Zolbanin and B.~Snapp, {\em A Note on the variety of pairs of matrices whose product is symmetric}, Cont. Math., \textbf{555} (2011).

\bibitem[Mu]{Mu:2007}
C. C.~Mueller, \href{http://www.math.lsa.umich.edu/hochster/mueller.thesis.pdf.}{{\em On the varieties of pairs of matrices whose product is symmetric}}, Ph.D. thesis, The University of Michigan, 2007. 

\bibitem[N]{N:2012}
N. V.~Ngo, {\em Cohomology for Frobenius kernels of $SL_2$}, to appear in Journal of Algebra.

\bibitem[Po]{Po:2008}

V.~L. Popov, {\em Irregular and singular loci of commuting varieties}, Transform. Groups, \textbf{13} (2008), 819--837.

\bibitem[Pr]{Pr:2003}

A.~Premet, {\em Nilpotent commuting varieties of reductive Lie algebras}, Invent. Math., \textbf{154} (2003), 653--683.

\bibitem[R]{R:1979}
R. W.~Richardson, {\em Commuting varieties of semisimple Lie algebras and algebraic groups}, Comp. Math., \textbf{38} (1979), 311--327.


\bibitem[S]{S:2012}
K.~Sivic, {\em On varieties of commuting triples III}, Lin. Alg. and App., \textbf{437} (2012), 393--460.



\bibitem[SB]{SB:1984}

J.~R. Schmidt and A.~M. Bincer, {\em The Kostant partition function for simple Lie algebras}, J. Math. Phys., \textbf{25} (1984), 2367--2374.

\bibitem[SFB1]{BFS1:1997}
A.~Suslin, E. M.~Friedlander, and C. P.~Bendel, {\em Infinitesimal 1-parameter subgroups and cohomology}, J. Amer. Math. Soc, \textbf{10} (1997), 693--728.

\bibitem[SFB2]{BFS2:1997}
\bysame, {\em Support varieties for infinitesimal group scheme}, J. Amer. Math. Soc, \textbf{10} (1997), 729--759.


\bibitem[TY]{TY:2005}
P.~Tauvel and R.~Yu, {\em Lie Algebras and Algebraic Groups}, Springer Monographs in Mathematics, Springer-Verlag, 2005.

\bibitem[Vas]{Vas:1994}
W. V.~Vasconcelos, {\em Arithmetic of Blowup Algebras}, Cambridge University Press, 1994.

\bibitem[Z]{Zoque:2010}
E.~Zoque, {\em On the variety of almost commuting nilpotent matrices}, Transform. Groups, \textbf{15} (2010), 483--501.

\end{thebibliography}
\end{document}